\documentclass[10pt,reqno]{amsart}
\usepackage{amsmath}
\usepackage{amssymb}
\usepackage{amsthm}
\usepackage{eepic,epic}
\usepackage{epsfig}
\usepackage{graphicx,color}

\newlength{\defbaselineskip}
\newcommand{\setlinespacing}[1]%
{\setlength{\baselineskip}{#1 \defbaselineskip}}

\numberwithin{equation}{section}

\newtheorem{thm}{Theorem}[section]

\newtheorem{lem}[thm]{Lemma}
\newtheorem{prop}[thm]{Proposition}

\theoremstyle{definition}

\theoremstyle{remark}
\newtheorem{rem}[thm]{Remark}
\numberwithin{equation}{section}

\begin{document}

\title[Strichartz estimates and local regularity]
{Strichartz estimates and local regularity for the elastic wave equation with singular potentials}

\author{Seongyeon Kim, Yehyun Kwon and Ihyeok Seo}

\thanks{This work was supported by a KIAS Individual Grant (MG073701) at Korea Institute for Advanced Study and NRF-2020R1F1A1A01073520 (Kwon) and NRF-2019R1F1A1061316 (Seo).}

\subjclass[2010]{Primary:35B45,35B65; Secondary:35L05 }
\keywords{Strichartz estimates, regularity, elastic wave equation}

\address{Department of Mathematics, Sungkyunkwan University, Suwon 16419, Republic of Korea}
\email{synkim@skku.edu}

\address{School of Mathematics, Korea Institute for Advanced Study, Seoul 02455, Republic of Korea}
\email{yhkwon@kias.re.kr}

\address{Department of Mathematics, Sungkyunkwan University, Suwon 16419, Republic of Korea}
\email{ihseo@skku.edu}

\begin{abstract}
We obtain weighted $L^2$ estimates for the elastic wave equation perturbed by singular potentials including the inverse-square potential.
We then deduce the Strichartz estimates under the sole ellipticity condition for the Lam\'e operator $-\Delta^\ast$. This improves upon the previous result in \cite{BFRVV} which relies on a stronger condition to guarantee the self-adjointness of $-\Delta^\ast$.
Furthermore, by establishing local energy estimates for the elastic wave equation we also prove that the solution has local regularity.
\end{abstract}

\maketitle

\section{Introduction}
We consider the Cauchy problem for the elastic wave equation
\begin{equation} \label{eq}
\begin{cases}
\partial_t^2 u -\Delta^* u + \delta V(x) u = 0, \\
u(0,x)=f(x), \quad \partial_t u(0,x)=g(x),
\end{cases}
\end{equation}
where the fields $f, g\colon \mathbb{R}^n \rightarrow \mathbb{R}^n$ and $u\colon \mathbb{R}\times \mathbb{R}^n \rightarrow \mathbb{R}^n$ take values in $\mathbb R^n$, $\delta\in\mathbb R$, $V\colon\mathbb{R}^n \rightarrow \mathcal{M}_{n\times n}(\mathbb{R})$ is a matrix valued potential, and $\Delta^*$ denotes the Lam\'e operator defined by
\begin{equation} \label{lameform}
	\Delta^*u =\mu\Delta u +(\lambda + \mu)\nabla \mathrm{div}u.
\end{equation}
Here the Laplacian $\Delta$ acts on each component of $u$ and
the Lam\'e coefficients $\lambda,\mu\in\mathbb{R}$ are assumed to obey the standard conditions
\begin{equation} \label{ellipticity}
\mu >0, \quad \lambda + 2\mu >0
\end{equation}
which guarantee the ellipticity of $\Delta^*$ (see \eqref{decLa} below).
The equation has been widely used to describe wave propagation in an elastic medium,
and in such case $u$ denotes the displacement field of the medium (e.g., \cite{LL,MH}).

When $\lambda+\mu=0$, the equation \eqref{eq} is reduced to the classical wave equation
\begin{equation} \label{wave_eqn}
\begin{cases}
\partial_t ^2u  - \Delta u + \delta V(x)u= 0, \\
u(0,x)=f(x),\quad \partial_t u(0,x) = g(x),
\end{cases}
\end{equation}
where $f$, $g$, $u$ and $V$ are now assumed scalar valued. The following space-time integrability of the solution to \eqref{wave_eqn}, known as \textit{Strichartz estimates}, has been extensively studied over the past several decades:
\begin{equation} \label{Strr}
\|u\|_{L^q_t \dot H^{\sigma}_r} \lesssim \|f\|_{\dot H^{1/2}} + \|g\|_{\dot H^{-1/2}},
\end{equation}
where $(q,r)$ is wave-admissible, i.e., $2\le q\le \infty$, $2\leq r < \infty$,
\begin{equation} \label{admi}
\frac{2}{q} + \frac{n-1}{r} \leq \frac{n-1}{2} \quad \text{and} \quad \sigma = \frac{1}{q} + \frac{n}{r} - \frac{n-1}{2}.
\end{equation}
When $n=3$  \eqref{Strr} fails at the endpoint $(q,r)=(2,\infty)$ (\cite{KM}). For the free wave equation ($\delta=0$), the Strichartz estimate \eqref{Strr} with $q=r$ was obtained in \cite{Str} in connection with the restriction theorems for the cone. See \cite{LS,KT} for the general case $q\neq r$.
For the perturbed case ($\delta\neq0$), \eqref{Strr} has been intensively studied to get as close as possible to potentials that decay like $|x|^{-2}$ for large $x$
(\cite{BS,C,GV,PST-Z,BPST-Z,BPST-Z2,KSS}).
This is because this decay is known to be critical for the validity of the Strichartz estimates \eqref{Strr}.
See \cite{GVV} which concerns explicitly the Schr\"odinger case but can be adapted to the wave equation as well.

While \eqref{Strr} is well understood for the wave equation, much less is known for the elastic case.
 Concerning \eqref{eq}, Barcel\'o et al. \cite{BFRVV} recently obtained \eqref{Strr} for small $|\delta|$
 and the wave-admissible $(q,r)$, $q>2$, with a symmetric $V$ having the critical\footnote{As remarked in \cite{BFRVV}, by the same kind of arguments as in \cite{GVV}, it may be shown that the inverse square decay is critical also for the elastic wave equation.} decay $|V(x)|\lesssim |x|^{-2}$,
but under the stronger condition
$$
	\mu >0, \quad \lambda + \frac{2\mu}{n} >0
$$
than the ellipticity condition \eqref{ellipticity}.
This strong assumption admits the self-adjoint (Friedrichs) extension of the positive symmetric operators $-\Delta^\ast$ and $-\Delta^* +\delta V$
defined initially on $H^1$.
By the spectral theorem, the solution to \eqref{eq} is now written as
\begin{equation*}
u= \cos (t\sqrt{-\Delta^*+\delta V})f + \frac{\sin(t\sqrt{-\Delta^*+\delta V})}{\sqrt{-\Delta^* +\delta V}} g
\end{equation*}
on which their work is based.

In this abstract approach, it seems not possible to get \eqref{Strr} with the ellipticity condition \eqref{ellipticity} only.
To overcome this lack, we instead focus on the Fourier multiplier of $\sqrt{-\Delta^*}$
and then represent the solution to \eqref{eq} as the sum of the solution to the free case plus a Duhamel term
by regarding the potential term as a source term.
Advantages of our approach are twofold: We do not need any assumption on $\lambda$ and $\mu$ other than \eqref{ellipticity},
and the potential $V$ does not need to be symmetric any more.

We also further relax the pointwise condition $|V(x)|\lesssim |x|^{-2}$ into some integrability condition at the same scale, namely, $V\in\mathcal{F}^p$
where $\mathcal F^p$ denotes the Fefferman-Phong class defined, for $1 \leq p \leq n/2$, by
\begin{equation*}
	\mathcal F^p:=\bigg\{ V\colon\mathbb R^n\to \mathcal M_{n\times n}(\mathbb R)\colon \|V\|_{\mathcal{F}^p} =\sup_{x\in \mathbb{R}^n,r>0} r^{2-\frac np} \bigg( \int_{B(x,r)}|V(y)|^p dy\bigg)^{\frac1p} < \infty\bigg\}.
\end{equation*}
Here, $B(x,r)$ denotes the open ball in $\mathbb R^n$ centered at $x$ with radius $r$ and $|V|=\big(\sum_{i,j=1}^n |V_{ij}|^2\big)^{1/2}$.
Note that if $1\leq p<n/2$ the class $\mathcal{F}^p$ contains the weak space $L^{n/2,\infty}$, and in particular, the inverse square potential $|x|^{-2}$.

To achieve the purpose we combine two kind of arguments; one from the theory of maximal functions involving the $A_p$ class and the other from elliptic regularity estimates (see Section \ref{sec4}).
As a result, we obtain the following theorem which improves upon \cite[Theorem 1.2]{BFRVV}.

\begin{thm} \label{thm1}
	Let $n\ge3$ and $V \in \mathcal{F}^p$ for $p>(n-1)/2$.
	 Let $u$ be a solution to \eqref{eq} with $(f,g)\in \dot H^{1/2} (\mathbb{R}^n) \times \dot H^{-1/2}(\mathbb{R}^n)$.
	There is an $\varepsilon >0$ such that if $|\delta|<\varepsilon$ then
	\begin{equation} \label{Str}
	\|u\|_{L_t^q \dot H^{\sigma}_{r}} \lesssim \|f\|_{\dot H^{1/2}} + \|g\|_{\dot H^{-1/2}}
	\end{equation}
whenever $(q,r)$ is wave-admissible with $q>2$.
Here, $\|u\|_{L^r} = \|\{u_j\}\|_{\ell_j^rL^r}$.
\end{thm}

In addition to the Strichartz estimates, we also obtain the following theorem which particularly shows a local regularity of solutions to the elastic wave equation. Let us denote by $L^2_{x,t}(w(x))$ the weighted space $L^2(w(x)dxdt)$.

\begin{thm} \label{thm2}
Let $n\ge3$ and $V \in \mathcal{F}^p$ for $p>(n-1)/2$.
There is an $\varepsilon>0$ such that if $|\delta|\le \varepsilon$ then there exists a unique solution $u \in L^{2}_{x,t}(|V|)$ to \eqref{eq}
with $(f,g)\in \dot H^{1/2} (\mathbb{R}^n) \times \dot H^{-1/2}(\mathbb{R}^n)$
	satisfying
	\begin{equation} \label{exi}
	u \in C([0,\infty);\dot H^{1/2}(\mathbb{R}^n)) \quad and \quad \partial_t u \in C([0,\infty);\dot H^{-1/2}(\mathbb{R}^n)).
	\end{equation}
	Furthermore,
	\begin{equation} \label{wei}
	\|u\|_{L^2_{x,t}(|V|)} \lesssim \|V\|^{1/2}_{\mathcal{F}^p} \big(\|f\|_{\dot H^{1/2}} + \|g\|_{\dot H^{-1/2}} \big)
	\end{equation}
	and for any $T>0$
\begin{align}\label{smoo}
	\nonumber\sup_{x_0\in \mathbb{R}^n ,R>0} \frac{1}{R} \int_{B(x_0,R)}\int_{-T}^{T} & |(1-\Delta)^{1/4}u|^2+|(1-\Delta)^{-1/4}\partial_tu|^2 dtdx \\
&\lesssim C(T)(\|f\|^2_{\dot H^{1/2}} + \|g\|^2_{\dot H^{-1/2}}).
\end{align}
\end{thm}

\begin{rem}
The estimate \eqref{smoo} can be regarded as some sort of \textit{local energy estimate}.
Indeed, it shows that the energy in any cylinder $B(x,R)\times [t_0, t_1] $  decreases uniformly like $\sim\sqrt R$ as $R\to 0$.
\end{rem}

The weighted $L^2$ estimate \eqref{wei} was already shown for the wave equation \eqref{wave_eqn} with $V(x)\sim|x|^{-2}$
and used to get the Strichartz estimates \eqref{Strr} directly without utilizing dispersive estimates (see \cite{BPST-Z,BPST-Z2}).
It was also shown in \cite{RV} that the wave equation \eqref{wave_eqn} has regularity locally as \eqref{smoo}.
Furthermore, as mentioned in \cite{D}, if $\rho$ is any function such that $\sum_{j \in \mathbb{Z}} \|\rho\|^2_{L^{\infty}(|x|\sim2^j)} <\infty$,
one has trivially
\begin{equation*}
\|\rho |x|^{-1/2} u_j\|^2_{L^2_{t}([-T,T];L_x^2(\mathbb{R}^n))} \lesssim \sup_{R>0} \frac{1}{R} \int_{|x|<R} \int_{-T}^{T} |u_j|^2 dtdx.
\end{equation*}
We can take, for example, $\rho(x)=|\log |x||^{-1/2-\delta}$ for any $\delta>0$.
Combining this with \eqref{smoo} gives the following global regularity for \eqref{eq} in terms of weighted $L^2$ norm:
\begin{align*}
\|(1-\Delta)^{1/4}u\|_{L^2_t([-T,T];L_x^2(\rho^2|x|^{-1}))}^2+&\|(1-\Delta)^{-1/4}\partial_tu\|_{L^2_t([-T,T];L_x^2(\rho^2|x|^{-1}))}^2\\
&\lesssim C(T)(\|f\|^2_{\dot H^{1/2}} + \|g\|^2_{\dot H^{-1/2}}).
\end{align*}

\

The rest of this paper is organized as follows.
In Section \ref{sec2}, we represent the solutions to inhomogeneous elastic wave equations utilizing the Fourier transform and provide some useful properties of the \textit{Helmholtz decomposition} by which the action of the Lam\'e operator on vector fields is simplified.
To prove the theorems in Sections \ref{sec5} and \ref{sec6},
 we write the solution to \eqref{eq} as a sum of the solution to the free case plus a Duhamel term regarding the potential term as a source term.
 Then we apply  the relevant weighted $L^2$ and local smoothing estimates obtained in Sections \ref{sec3} and \ref{sec4} to each of the terms.

\subsubsection*{Acknowledgment}
The authors would like to thank the anonymous referee for some valuable comments and suggestions on
an issue about the theory of wave operators by K. Yajima \cite{K}.

\section{Preliminaries} \label{sec2}
\subsection*{Representation of the solution}
In this section, we first give the solution to the following inhomogeneous elastic wave equation in terms of Fourier multipliers.
\begin{equation} \label{inhoeq}
\begin{cases}
	\partial_t^2u(t,x) -\Delta^*u(t,x)  = F(t,x), \\
	u(0,x)=f(x), \quad \partial_t u(0,x)=g(x).
\end{cases}
\end{equation}

Taking the Fourier transform in the spatial variable, we convert \eqref{inhoeq} into the following system of ODE
\begin{equation} \label{Feq}
\begin{cases}
	\partial_t^2  \widehat{u} = -L  \widehat{u} +  \widehat{F}, \\
	 \widehat{u} (0) =  \widehat{f} , \quad \partial_t  \widehat {u}(0) =  \widehat{g},
\end{cases}
\end{equation}
where $L=L(\xi)=\mu |\xi|^2 I_n + (\lambda +\mu) (\xi \xi^t)$ is the (matrix) symbol of $-\Delta^\ast$.
Since $\mu >0$ and $\lambda +2\mu>0$, for every non-zero $\xi$, the matrix $L$ is positive-definite. Indeed, for $z\in \mathbb C^n\setminus\{0\}$, if $z^\ast \xi=0$ then $z^\ast L(\xi)z=\mu |\xi|^2 |z|^2>0$ and if $z^\ast \xi\neq 0$ then
\[	z^\ast  L(\xi)z  \\
	= \mu |\xi|^2 |z|^2 + (\lambda+\mu) |z^\ast  \xi |^2
	\ge (\lambda+2\mu) |z^\ast \xi |^2 >0
\]
by the Schwarz inequality.
Hence there exists a unique positive-definite square root $\sqrt{L}$.
Let us denote by $\sqrt{-\Delta^*}$ and $\sqrt{-\Delta^*}^{-1}$ the Fourier multiplier operators defined by the multipliers $\sqrt{L}$ and $\sqrt{L}^{-1}$, respectively.

By a standard method of reducing \eqref{Feq} to a first-order system and Duhamel's principle, if we set
$$	{\bf v}=( \widehat{u}, \partial_t \widehat {u})^t, \quad
	{\bf b}=( 0,  \widehat{F})^t   \quad
	\text{and}\quad
	{\bf A}=	\begin{pmatrix}
		0	& I_n \\
		-L	& 0
		\end{pmatrix},	$$
we then have
$$	{\bf v}(t)= e^{t{\bf A}} {\bf v}(0) +\int_0^t e^{(t-s){\bf A}} {\bf b}(s)ds,	$$
where $e^{M}=\sum_{k=0}^\infty M^k/k!$ for any square matrix $M$.  Since $e^{\theta}=\cos(-i\theta)+i\sin(-i\theta)$ the first $n$-component of ${\bf v}$ is given by
$$ \widehat {u}(t)= \cos(t\sqrt{L}) \widehat{f} +\sin(t\sqrt{L})\sqrt{L}^{-1}  \widehat {g} +\int_0^t \sin\big((t-s)\sqrt{L}\big)\sqrt{L}^{-1} \widehat {F}(s)ds.$$
Now taking the inverse Fourier transform, the solution to \eqref{inhoeq} is written as
\begin{equation*}
u(t) = \cos(t\sqrt{-\Delta^*}) f + \sin (t\sqrt{-\Delta^*})\sqrt{-\Delta^*}^{-1}g + \int_0^t \sin\big((t-s)\sqrt{-\Delta^*}\big)\sqrt{-\Delta^*}^{-1} F(s)ds,
\end{equation*}
where $\varphi(\sqrt{-\Delta^*})$ denotes the Fourier multiplier operator defined by the symbol $\varphi(\sqrt L)$.

\subsection*{The Helmholtz decomposition}
Next we are concerned with some useful properties of the so called Helmholtz decomposition and related consequences
on the Lam\'e operators and Fourier multipliers.

The action of the Lam\'e operator $\Delta^*$ on vector fields is well understood using the \emph{Helmholtz decomposition}, which states that
every $f\in [L^2(\mathbb R^n)]^n$ can be uniquely decomposed as
$$	f=f_S+f_P,	$$
where $\mathrm{div} f_S=0$ and $f_P=\nabla \varphi$ for some\footnote{In fact, $\varphi$ is a solution of the Poisson equation $\Delta\varphi=\mathrm{div} f$.}
$\varphi\in \dot{H}^1(\mathbb{R}^n)$.
The components $f_S$ and $f_P$ are orthogonal in the sense that $\langle f_S, f_P\rangle_{[L^2(\mathbb R^n)]^n}=0$, where
$$	\langle f,g\rangle_{[L^2(\mathbb R^n)]^n}=\sum_{j=1}^n \langle f_j, g_j\rangle_{L^2(\mathbb R^n)}
$$
is the inner product on $[L^2(\mathbb R^n)]^n$.
For a proof we refer the reader to \cite[pp. 81--83]{Sohr}.
From now on, we will abuse notation and simply write $\langle f, g\rangle _{L^2}$ for this inner product.

By the decomposition it follows that
\begin{equation} \label{decLa}
\Delta^*u = \mu \Delta u_S +(\lambda + 2\mu)\Delta u_P,
\end{equation}
which show that $\Delta^*$ is elliptic under the condition \eqref{ellipticity}.
Unlike the form \eqref{lameform} that is far from being easy to handle,
this representation \eqref{decLa} gives us a possibility of reducing the matter at first
to the level of the Laplacian.
Indeed, by the representation combined with the orthogonality, the elastic wave equation \eqref{inhoeq}
is split into the following system of two wave equations involving separately the vector fields $u_S$ and $u_P$:
\begin{equation} \label{decoup}
	\begin{cases}
		\partial_t^2 u_S - \mu \Delta u_S = F_S, \\
		u_S(0)=f_S, \ \  \partial_t u_S(0)=g_S,
	\end{cases}
\quad
	\begin{cases}
		\partial_t^2 u_P - (\lambda +2\mu) \Delta u_P = F_P, \\
		u_P (0)=f_P, \ \  \partial_t u_P(0)= g_P.
	\end{cases}
\end{equation}
In the proofs of our main results, the final steps after obtaining separated results on $u_S$ and $u_P$ will be to recombine those together in order to get desired results for $u$.
This will be done by developing some non-trivial techniques in later sections.

Another favorable property of the Helmholtz decomposition is that the projectors $f\mapsto f_S$ and $f\mapsto f_P$ (called \textit{Leray projectors}) are in fact Fourier multipliers with matrix valued symbols. This leads to commutativity between the Leray projectors and certain Fourier multipliers $M(D)$ defined by $M(D) f =(M(\xi) \widehat f \,)^\vee$, and as such the $L^2$-orthogonality is available for $M(D)f_S$ and $M(D)f_P$.

\begin{lem}\label{Ortho}
Let $M(\xi)$ be an $n\times n$ matrix valued tempered function defined on $\mathbb R^n$ and let $f\in [L^2(\mathbb R^n)]^n$.  Then we have
\begin{equation}\label{commute}
	(M(D)f)_S=M(D)f_S, \quad (M(D)f)_P=M(D)f_P
\end{equation}
if and only if $M(\xi)=m(\xi)I_n$ for some scalar function $m$. In this case, $m(D)f_S$ and $m(D)f_P$ are orthogonal in $[L^2(\mathbb R^n)]^n$ and
\begin{equation}\label{Pythagorean}
	\|m(D)f\|_{L^2}^2 = \|m(D)f_S\|_{L^2}^2 + \|m(D)f_P\|_{L^2}^2.
\end{equation}
\end{lem}
\begin{proof}
Set $f=f_S+f_P=(f-\nabla\varphi)+\nabla\varphi$ for $\varphi\in H^1(\mathbb{R}^n)$.
Taking the Fourier transform on the equation $\mathrm{div} f=\mathrm{div} \nabla\varphi$, it is easy to see that $\widehat\varphi(\xi)=-i|\xi|^{-2}\xi^t\widehat f(\xi)$.
Hence we get
$$ \widehat{f_P}(\xi)=\widehat{\nabla\varphi}(\xi)=i\xi\widehat\varphi(\xi) =|\xi|^{-2}\xi\xi^t \widehat f(\xi)\quad \text{and} \quad \widehat{f_S}(\xi)=(I_n-|\xi|^{-2}\xi\xi^t) \widehat f(\xi),
$$
and it follows that \eqref{commute} holds if and only if $\xi\xi^tM(\xi)=M(\xi)\xi\xi^t$ for all $\xi\in\mathbb R^n$, which is equivalent to $M_{ij}(\xi)=m(\xi)\delta_{ij}$ for some $m$.

Once we have \eqref{commute}, the Pythagorean identity \eqref{Pythagorean} is a direct consequence of the aforementioned $L^2$-orthogonality of the Leray projectors.
\end{proof}

\section{Estimates for the free propagators} \label{sec3}
Now we obtain the weighted $L^2$ and local smoothing estimates for the propagators $e^{it\sqrt{-\Delta^*}}$ and $e^{it\sqrt{-\Delta^*}}\sqrt{-\Delta^*}^{-1}$
that constitute the solution to the homogeneous problem \eqref{inhoeq} with $F=0$.
These estimates will be used in the next sections for the proofs of Theorems \ref{thm1} and \ref{thm2}.
The former is firstly stated as follows.
\begin{prop} \label{prop1}
Let $n \ge 3$ and $V \in \mathcal{F}^p$ for $p>(n-1)/2$. Then we have
\begin{equation} \label{weihomo}
\|e^{it\sqrt{-\Delta^*}}f \|_{L_{x,t}^2(|V|)} \lesssim \|V\|^{1/2}_{\mathcal{F}^p} \|f\|_{\dot H^{1/2}}
\end{equation}
and
\begin{equation} \label{weihomo'}
\|e^{it\sqrt{-\Delta^*}}\sqrt{-\Delta^*}^{-1} g \|_{L_{x,t}^2(|V|)} \lesssim\|V\|^{1/2}_{\mathcal{F}^p} \|g\|_{\dot H^{-1/2}}.
\end{equation}
\end{prop}

\begin{proof}
Once we obtain the following norm equivalence
\begin{equation}\label{normequi}
\|\sqrt{-\Delta^*}f\|_{L^2} \sim \|f\|_{\dot H^1}
\end{equation}
the estimates \eqref{weihomo} and \eqref{weihomo'} are easy consequences of the estimates
\begin{equation} \label{cosW}
\|\cos(t\sqrt{-\Delta^*})f \|_{L_{x,t}^2(|V|)} \lesssim \|V\|^{1/2}_{\mathcal{F}^p} \|f\|_{\dot H^{1/2}}
\end{equation}
and
\begin{equation} \label{sinW}
\|\sin(t\sqrt{-\Delta^*})\sqrt{-\Delta^*}^{-1} g \|_{L_{x,t}^2(|V|)} \lesssim\|V\|^{1/2}_{\mathcal{F}^p} \|g\|_{\dot H^{-1/2}}
\end{equation}
which will be shown later.
Indeed, assuming \eqref{normequi} for the moment, \eqref{sinW} gives
\[	\|\sin(t\sqrt{-\Delta^*})f \|_{L_{x,t}^2(|V|)} \lesssim \|V\|^{1/2}_{\mathcal{F}^p} \|\sqrt{-\Delta^*} f\|_{\dot H^{-1/2}} \sim \|V\|^{1/2}_{\mathcal{F}^p} \|f\|_{\dot H^{1/2}}.	\]
Combining this with \eqref{cosW}, we get \eqref{weihomo}.
Moreover, setting $f=\sqrt{-\Delta^*}^{-1} g$ in the inequality \eqref{weihomo} and using \eqref{normequi} again, we have
\[	\|e^{it\sqrt{-\Delta^*}}\sqrt{-\Delta^*}^{-1} g \|_{L_{x,t}^2(|V|)} \lesssim \|V\|^{1/2}_{\mathcal{F}^p} \|\sqrt{-\Delta^*}^{-1} g\|_{\dot H^{1/2}} \sim \|V\|^{1/2}_{\mathcal{F}^p} \|g\|_{\dot H^{-1/2}}.	\]

Now we show \eqref{normequi}. Since $\sqrt{L}$ is self-adjoint, Parseval's identity gives
\[	\|\sqrt{-\Delta^*}f\|^2_{L^2}
= (2\pi)^{-n} \langle \sqrt{L} \hat f , \sqrt{L} \hat f \,\rangle_{L^2}
= (2\pi)^{-n} \langle \hat f, L \hat f \, \rangle_{L^2}
=\langle f, -\Delta^* f\rangle_{L^2}.
\]
By the Helmholtz decomposition, \eqref{decLa} and Lemma \ref{Ortho}, we then have
\begin{align*}
\langle f, -\Delta^* f\rangle_{L^2}
&= \langle f_P + f_S , -\mu \Delta f_S - (\lambda + 2\mu)\Delta f_P \rangle_{L^2} \\
&= \mu \int_{\mathbb{R}^n} |\nabla f_S|^2 dx+ (\lambda +2\mu) \int_{\mathbb{R}^n} {|\nabla f_P|^2} dx
\sim \|\nabla f\|_{L^2}^2 \sim \|f\|_{\dot H^1}^2.
\end{align*}

It remains to prove \eqref{cosW} and \eqref{sinW}.
To show \eqref{cosW}, we first set $F=g=0$ in the equation \eqref{inhoeq}.
The solution is then written as the sum of the corresponding solutions to the couple of equations \eqref{decoup}:
\begin{equation}\label{cosine}
\cos (t\sqrt{-\Delta^\ast})f =\cos (t\sqrt{-\mu\Delta})f_S + \cos (t\sqrt{-(\lambda+2\mu)\Delta})f_P.
\end{equation}
Hence, assuming the following estimates
\begin{gather}
\label{weihoS}	\|\cos (t\sqrt{-\mu\Delta}) f_S \|_{L_{x,t}^2(|V|)} \lesssim \|V\|^{1/2}_{\mathcal{F}^p}  \|f_S\|_{\dot H^{1/2}}, \\
\label{weihoP}	\|\cos (t\sqrt{-(\lambda+2\mu)\Delta}) f_P \|_{L_{x,t}^2(|V|)} \lesssim \|V\|^{1/2}_{\mathcal{F}^p}  \|f_P\|_{\dot H^{1/2}}
\end{gather}
for the moment and
making use of the orthogonality \eqref{Pythagorean}, we have
\[	\|\cos (t\sqrt{-\Delta^\ast})f\|_{L_{x,t}^2(|V|)}
\lesssim \|V\|^{1/2}_{\mathcal{F}^p} \big( \|f_S\|_{\dot H^{1/2}} + \|f_P\|_{\dot H^{1/2}} \big)
\lesssim \|V\|^{1/2}_{\mathcal{F}^p} \|f\|_{\dot H^{1/2}},
\]
which gives \eqref{cosW}.
Now we need only show \eqref{weihoS} since the proof of \eqref{weihoP} is similar.
But this follows from applying the known estimate \eqref{cosW} with $\Delta^*$ replaced by the Laplacian $\Delta$ (see (2.11) in \cite{RV}):
\begin{align*}
	\| \cos(t\sqrt{-\mu\Delta}) f_S \|_{L^2_{x,t}(|V|)}
	&= \bigg( \sum_{j=1}^n \| \cos(t\sqrt{-\mu\Delta})(f_S)_j \|^2_{L^2_{x,t}(|V|)} \bigg)^{1/2} \\
	&\lesssim \bigg( \|V\|_{\mathcal{F}^p}\sum_{j=1}^n \|(f_S)_j\|^2_{\dot H^{1/2}} \bigg)^{1/2}
	= \|V\|^{1/2}_{\mathcal{F}^p}\|f_S\|_{\dot H^{1/2}}.
\end{align*}
Similarly we prove \eqref{sinW}; first set $F=f=0$ in \eqref{inhoeq}, and
look at \eqref{decoup} to see that
\begin{equation}\label{sine}
	\begin{aligned}
	&\sin(t\sqrt{-\Delta^*})\sqrt{-\Delta^*}^{-1} g \\
	&\quad= \sin(t\sqrt{-\mu\Delta})\sqrt{-\mu\Delta}^{-1} g_S
	+ \sin(t\sqrt{-(\lambda+2\mu) \Delta})\sqrt{-(\lambda+2\mu)\Delta}^{-1} g_P.
	\end{aligned}
\end{equation}
On the right side of \eqref{sine} we apply the known estimate \eqref{sinW} for the Laplacian $\Delta$ (see (2.12) in \cite{RV})
and the orthogonality \eqref{Pythagorean} to get \eqref{sinW}.
\end{proof}

Next we obtain the following local smoothing estimates using the same argument employed to prove Proposition \ref{prop1}.

\begin{prop} \label{prop2}
Let $n \ge 3$. Then we have
\begin{equation}\label{smoohom}
\sup_{x_0 \in \mathbb{R}^n, R>0} \frac{1}{R} \int_{B(x_0, R)}\int_{-\infty}^{\infty} \big| |\nabla|^{1/2} e^{it \sqrt{-\Delta^*}}f \big|^2 dtdx \lesssim \|f\|^2_{\dot H^{1/2}}
\end{equation}
and
\begin{equation}\label{smoohom'}
\sup_{x_0 \in \mathbb{R}^n, R>0} \frac{1}{R} \int_{B(x_0,R)}\int_{-\infty}^{\infty} \big| |\nabla|^{1/2} e^{it \sqrt{-\Delta^*}}
\sqrt{-\Delta^*}^{-1} g \big|^2 dtdx \lesssim \|g\|^2_{\dot H^{-1/2}}.
\end{equation}
\end{prop}

\begin{proof}
Making use of \eqref{normequi}, the estimates \eqref{smoohom} and \eqref{smoohom'} are easy consequences of the following estimates
\begin{equation}\label{cosmoo}
\sup_{x_0, R} \frac{1}{R} \int_{B(x_0, R)}\int_{-\infty}^{\infty} \big| |\nabla|^{1/2} \cos(t \sqrt{-\Delta^*})f \big|^2 dtdx \lesssim \|f\|^2_{\dot H^{1/2}}
\end{equation}
and
\begin{equation}\label{sismoo}
\sup_{x_0, R} \frac{1}{R} \int_{B(x_0,R)}\int_{-\infty}^{\infty} \big| |\nabla|^{1/2} \sin(t \sqrt{-\Delta^*})
\sqrt{-\Delta^*}^{-1} g \big|^2 dtdx \lesssim \|g\|^2_{\dot H^{-1/2}}.
\end{equation}
Indeed, by using \eqref{normequi}, \eqref{sismoo} gives
\begin{equation*}
\sup_{x_0 , R} \frac{1}{R} \int_{B(x_0,R)}\int_{-\infty}^{\infty} \big| |\nabla|^{1/2} \sin(t \sqrt{-\Delta^*}) f \big|^2 dtdx \lesssim \|\sqrt{-\Delta^*}f\|^2_{\dot H^{-1/2}} \sim \|f\|^2_{\dot H^{1/2}}.
\end{equation*}
Combining this with \eqref{cosmoo}, we have \eqref{smoohom}.
Also, using \eqref{smoohom} with $f=\sqrt{-\Delta^*}^{-1} g$ and \eqref{normequi}, we get
\begin{equation*}
\sup_{x_0 , R} \frac{1}{R} \int_{B(x_0, R)}\int_{-\infty}^{\infty} \big| |\nabla|^{1/2} e^{it \sqrt{-\Delta^*}}\sqrt{-\Delta^*}^{-1} g \big|^2 dtdx
\lesssim \|\sqrt{-\Delta^*}^{-1} g\|^2_{\dot H^{1/2}} \sim \|g\|^2_{\dot H^{-1/2}}.
\end{equation*}

Finally, it remains to show \eqref{cosmoo} and \eqref{sismoo}.
By \eqref{cosine} and the known estimate \eqref{cosmoo} with $\Delta^*$ replaced by $\Delta$ (see (2.21) in \cite{RV}), we have
$$	\sup_{x_0, R} \frac{1}{R} \int_{B(x_0, R)}\int_{-\infty}^{\infty} \big| |\nabla|^{1/2} \cos(t \sqrt{-\Delta^*})f \big|^2 dtdx
	\lesssim \|f_S\|^2_{\dot H^{1/2}} + \|f_P\|^2_{\dot H^{1/2}},
$$
and therefore \eqref{cosmoo} follows from Lemma \ref{Ortho}.
The second estimate \eqref{sismoo} follows in a similar manner using \eqref{sine} and the corresponding estimate (2.22) in \cite{RV}
for the Laplacian $\Delta$.
\end{proof}

\section{Inhomogeneous estimates} \label{sec4}
In addition to the estimates in the previous section, we need to obtain the corresponding estimates
for the inhomogeneous problem \eqref{inhoeq} with the zero initial data $f=g=0$.
Analogously to the preceding propositions, we will achieve this by reducing the matter at first
to the level of the Laplacian,
and then recombining the resulting respective estimates for $u_S$ and $u_P$ to yield the estimates for the solution $u$.
But this recombination is more delicate in the inhomogeneous case and requires some non-trivial techniques that consist of two kind of arguments; one from the theory of maximal functions involving the $A_p$ class and the other from elliptic regularity estimates.
Even if these are relatively well-known separately, their combination seems new in the present context.

\subsection*{Elliptic regularity estimates and maximal functions}
As a preliminary step, we first obtain the following elliptic regularity estimates (\textit{cf}. \cite{BFRVV,Co})
combined with the theory of maximal functions and $A_p$ class for weighted estimates.
\begin{lem} \label{lemma}
Let $n \ge 3$ and $V \in \mathcal{F}^p$ for $1<p\le n/2$.
For any $n$-tuple of tempered distributions $F=(F_1, \ldots, F_n) \in [\mathcal S'(\mathbb{R}^n)]^n$, if
\begin{equation} \label{poi}
	-\Delta \psi = {\rm div} F
\end{equation}
then we have
\begin{equation} \label{lem}
	\|\nabla \psi\|_{L^2(W^{\pm1})} \lesssim
	\|F\|_{L^2(W^{\pm1})}
\end{equation}
for $W = M(|V|^{\delta})^{1/\delta}$ with $1<\delta<p$.
Here, $M(f)$ denotes the Hardy-Littlewood maximal function of $f$ and the implicit constant is independent of $V$.
\end{lem}

Before we prove the lemma, let us record some facts about the weights.
Firstly, a nonnegative locally integrable function $w\colon \mathbb R^n\to [0,\infty]$
is called an $A_1$ weight whenever there is a constant $c$ such that
\[	Mw(x)\le c w(x)		\]
almost everywhere or, equivalently,
\[	[w]_{A_1}:=\sup_{B\,\, \text{balls in} \,\,\mathbb{R}^n} \bigg( \frac{1}{|B|} \int_{B} w (x) dx \bigg) \|w^{-1}\|_{L^\infty(B)} <\infty.	\]
For $1<p<\infty$, $w$ is also said to be of class $A_p$ if
\[	[w]_{A_p}:=\sup_{B\,\, \text{balls in} \,\,\mathbb{R}^n} \bigg( \frac{1}{|B|} \int_{B} w (x) dx \bigg)  \bigg(\frac{1}{|B|} \int_B w(x)^{-\frac{1}{p-1}} dx \bigg)^{p-1} < \infty.	\]
It is easy to see that if $1\le p<q<\infty$ then $A_p\subset A_q$ with $[w]_{A_q}\le [w]_{A_p}$.

Secondly, if $1<p\leq n/2$ and $V \in \mathcal{F}^{p}$, then for any $1< \delta < p$
\begin{equation*}
W=M(|V|^\delta)^{1/\delta} \in A_1 \cap \mathcal{F}^{p}
\end{equation*}
and
\begin{equation} \label{WV}
\|W\|_{\mathcal{F}^p} \lesssim \|V\|_{\mathcal{F}^p}.
\end{equation}
See \cite[Lemma 1]{CF} for the proof of this useful property of the Fefferman-Phong class.

Finally, the following lemma, which is taken from \cite[pp. 214--215]{S} (see also \cite[Proposition 2]{CR}), enables us to make the implicit constant in \eqref{lem} independent of $V$.
\begin{lem}\label{uniformA1}
	Let $\delta>1$. For any nonnegative function $w$, if $Mw <\infty$ almost everywhere, then $(Mw)^{1/\delta}\in A_1$ and $[w]_{A_1}$ is bounded by a constant independent of $w$.
\end{lem}

\begin{proof}[Proof of Lemma \ref{lemma}]
Taking the Fourier transform on \eqref{poi} it is easy to see that
\begin{equation}\label{Rzpsi}
	\partial_j\psi = \sum_{k=1}^n R_jR_kF_k,
\end{equation}
where $R_j$ is the Riesz transform defined by $ \widehat{R_j\varphi}(\xi) =i\xi_j|\xi|^{-1}  \widehat{\varphi} (\xi)$.
We shall then use the weighted $L^2$ estimate
\begin{equation}\label{rieszw}
	\|R_jf\|_{L^2(w)}\lesssim [w]_{A_2} \|f\|_{L^2(w)},
\end{equation}
which holds whenever $w\in A_2$ (see \cite{P}).

Since $W^{\pm1}\in A_2$ and $[W^{-1}]_{A_2}=[W]_{A_2}\le[W]_{A_1}$, it follows now from \eqref{Rzpsi}, \eqref{rieszw} and Minkowski's inequality that
$$	\|\nabla \psi\|^2_{L^2(W^{\pm1})}
	\le \sum_{j=1}^n \bigg( \sum_{k=1}^n \|R_j R_k F_k\|_{L^2(W^{\pm1})} \bigg)^2
	\lesssim
	[W]^2_{A_1} \|F\|^2_{L^2(W^{\pm1})}.
$$
Since $V\in \mathcal F^p$ and $1<\delta<p$, it is easy to see that $M(|V|^\delta)(x)<\infty$ for every Lebesgue point $x$ of $|V|^\delta$. Hence, $[W]_{A_1}$ is uniformly bounded by Lemma \ref{uniformA1}. \end{proof}

\subsection*{Inhomogeneous estimates}
Now we are ready to obtain the following desired estimates in the rest of this section.

\begin{prop} \label{prop3}
	Let $n \ge 3$ and $V \in \mathcal{F}^p$ for $p>(n-1)/2$. If $u$ is a solution to \eqref{inhoeq} with $f=g=0$, then we have
\begin{equation} \label{weiinho}
	\|u\|_{L^2_{x,t}(|V|)} \lesssim \|V\|_{\mathcal{F}^p}\|F\|_{L^2_{x,t}(|V|^{-1})}
\end{equation}
and
	\begin{equation} \label{smooinho}
	\sup_{x_0 \in \mathbb{R}^n, R>0} \frac{1}{R} \int_{B(x_0,R)}\int_{-\infty}^{\infty} \big| |\nabla|^{1/2}u\big|^2 dtdx \lesssim \|V\|_{\mathcal{F}^p}\|F\|^2_{L^2_{x,t} (|V|^{-1})}.
	\end{equation}
\end{prop}

\begin{proof}
Let $1<\delta<p$ and $W=M(|V|^\delta)^{1/\delta}$.
Since $|V|\le W$ almost everywhere, using \eqref{WV} we may prove the proposition by replacing $V$ with $W$.
The motivation behind this replacement is the availability of the elliptic regularity estimates, Lemma \ref{lemma}.

By \eqref{decoup} with $f=g=0$, we have
\begin{align*}
	u_S(t)&=\int_0^t \sin\big((t-s)\sqrt{-\mu\Delta}\big) \sqrt{-\mu\Delta}^{-1} F_S(s) ds, \\
	\nonumber	u_P(t)&=\int_0^t \sin\big((t-s)\sqrt{-(\lambda+2\mu) \Delta})\big) \sqrt{-(\lambda+2\mu) \Delta}^{-1} F_P(s) ds.
\end{align*}
Let us first show \eqref{weiinho} for $(u_S, F_S)$ and $(u_P, F_P)$, respectively, in place of $(u,F)$.
Applying the following inequality (\cite[Proposition 4.2]{RV}) for $w\in \mathcal F^p$, $p>(n-1)/2$,
\begin{equation}\label{kkk}
	\bigg\|\int_0^t \frac{\sin((t-s)\sqrt{-\Delta})}{\sqrt{-\Delta}} \varphi(s) ds  \bigg\|_{L^2_{x,t}(|w|)} \lesssim \|w\|_{\mathcal{F}^p}\|\varphi\|_{L^2_{x,t}(|w|^{-1})},
\end{equation}
to each component of $u_S$, we get
$$	\|u_S\|_{L^2_{x,t}(W)}
	\lesssim \bigg(\|W\|^2_{\mathcal{F}^p} \sum_{j=1}^n \|(F_S)_j\|^2_{L^2_{x,t}(W^{-1})}\bigg)^{1/2}
	=\|W\|_{\mathcal{F}^p}\|F_S\|_{L^2_{x,t}(W^{-1})}.
$$
The proof of \eqref{weiinho} for $(u_P, F_P)$ is similar. Hence we have
\begin{equation}\label{bfRV2}
	\|u\|_{L^2_{x,t}(W)}
	\lesssim \|W\|_{\mathcal{F}^p} \big( \|F_S\|_{L^2_{x,t}(W^{-1})} + \|F_P\|_{L^2_{x,t}(W^{-1})}\big).
\end{equation}
Finally we use Lemma \ref{lemma} to conclude
$$\|F_S\|_{L^2_{x,t}(W^{-1})} + \|F_P\|_{L^2_{x,t}(W^{-1})}\lesssim\|F\|_{L^2_{x,t}(W^{-1})}.$$
Fixing $t\in \mathbb R$ let us write $F_t(x)=F(x,t)$ and
\[	F_t=F_t-\nabla \psi_t+\nabla \psi_t,	\]
where $\psi_t$ is a solution to $-\Delta \psi_t =\mathrm{div} F_t$. Since $\mathrm{div}(F_t +\nabla \psi_t)=0$, by the uniqueness of the Helmholtz decomposition, we conclude that
\begin{equation*}
	(F_t)_S=F_t+\nabla \psi_t, \quad (F_t)_P =-\nabla \psi_t.
\end{equation*}
Making use of Lemma \ref{lemma}, we now see
\begin{align*}
	\|F_S\|_{L_{x,t}^2(W^{-1})} + \|F_P\|_{L_{x,t}^2(W^{-1})}
	&= \big\| \|F_t+\nabla \psi_t\|_{L^2_x(W^{-1})} \big\|_{L_t^2} +\big\| \|\nabla \psi_t\|_{L^2_x (W^{-1})} \big\|_{L_t^2} \\
	&\lesssim \|F\|_{L^2_{x,t}(W^{-1})}.
\end{align*}
Combining this with \eqref{bfRV2}, we get
\[	\|u\|_{L^2_{x,t}(W)} \lesssim \|W\|_{\mathcal{F}^p} \|F\|_{L^2_{x,t}(W^{-1})}	\]
as desired.

The proof of \eqref{smooinho} is similar replacing \eqref{kkk} by the following inequality ((1.4) in \cite[Theorem 1]{RV}),
\[	\sup_{x_0,R}\frac{1}{R} \int_{B(x_0,R)}\int_{-\infty}^{\infty} \Big| |\nabla|^\frac{1}{2}\int_0^t \frac{\sin((t-s)\sqrt{-\Delta})}{\sqrt{-\Delta}}G(s)ds\Big|^2  dtdx
	\lesssim \|w\|_{\mathcal{F}^p}\|G\|^2_{L^2_{x,t}(|w|^{-1})}.
\]
So we shall omit the details.
\end{proof}

\section{Proof of Theorem \ref{thm2}} \label{sec5}
This section is devoted to proving Theorem \ref{thm2}.
We first write the solution to \eqref{eq} as the sum of the solution to the free elastic wave equation ($\delta=0$) plus a Duhamel term
by considering the potential term as a source term.
Then we apply the estimates obtained in the previous sections to each of the terms.

\subsubsection*{Existence and uniqueness} For any $V\in\mathcal F^p$ in Theorem \ref{thm1} and  $\delta>0$, define
\begin{equation*}
	\mathcal{T}u (t) := \int_0^t \sin\big((t-s)\sqrt{-\Delta^*}\big) \sqrt{-\Delta^\ast}^{-1} \delta V(\cdot)u(\cdot,s)ds.
\end{equation*}
Regarding the potential term $\delta V(x)u(x,t)$ in \eqref{eq} as a source term ($F$ in \eqref{inhoeq}),
we see that if the solution $u$ exists then it must satisfy
\begin{equation} \label{sol}
	u=\mathcal A u(t):=\cos(t\sqrt{-\Delta^*}) f + \sin(t\sqrt{-\Delta^*}) \sqrt{-\Delta^*}^{-1} g + \mathcal{T}u(t).
\end{equation}
By Proposition \ref{prop1} and the estimate \eqref{weiinho} with $F=\delta Vu$,
the operator $\mathcal A\colon L^2_{x,t}(|V|)\to L^2_{x,t}(|V|)$ is well-defined. Moreover,
\begin{equation}\label{exiinho}
	\|\mathcal Tu\|_{L_{x,t}^2(|V|)} \lesssim \|V\|_{\mathcal F^p} \|\delta Vu\|_{L^2_{x,t}(|V|^{-1})} \leq|\delta| \|V\|_{\mathcal F^p} \|u\|_{L^2_{x,t}(|V|)}.
\end{equation}
Hence  $\|\mathcal Au_1-\mathcal Au_2\|_{L^2_{x,t}(|V|)}=\|\mathcal T(u_1- u_2)\|_{L^2_{x,t}(|V|)} \lesssim |\delta| \|V\|_{\mathcal F^p} \|u_1-u_2\|_{L^2_{x,t}(|V|)}$, that is, $\mathcal A$ is a contraction provided that $|\delta|$ is small enough. Thus, by the contraction mapping principle, there exists a unique solution $u\in L^2_{x,t}(|V|)$ to \eqref{eq}.

\subsubsection*{Proof of \eqref{exi}}
In order to prove $u \in C([0,\infty);\dot H^{1/2}(\mathbb{R}^n))$ it is enough to show
\begin{gather}
	\label{unicos}	\sup_{t\in\mathbb R} \| \cos(t\sqrt{-\Delta^\ast}) f \|_{\dot H^{1/2}} <\infty, \\
	\label{unisin}	\sup_{t\in\mathbb R} \| \sin(t\sqrt{-\Delta^\ast}) \sqrt{-\Delta^\ast}^{-1} g\|_{\dot H^{1/2}} <\infty
\end{gather}
and
\begin{equation}\label{uni_inho}
	\sup_{t\in\mathbb R} \bigg\| \int_0^t \sin\big((t-s)\sqrt{-\Delta^*}\big) \sqrt{-\Delta^\ast}^{-1} \delta V(\cdot)u(\cdot,s)ds \bigg\|_{\dot H^{1/2}} <\infty.
\end{equation}

The estimate \eqref{unicos} follows from the decomposition \eqref{cosine} combined with Plancherel's theorem and the orthogonality (Lemma \ref{Ortho}). Indeed, for any $t\in\mathbb R$,
\begin{align*}
	\| \cos(t\sqrt{-\Delta^\ast}) f \|_{\dot H^{1/2}}
	&\le \| \cos(t\sqrt{-\mu\Delta}) f_S \|_{\dot H^{1/2}} + \| \cos(t\sqrt{-(\lambda+2\mu)\Delta}) f_P \|_{\dot H^{1/2}} \\
	&\le \| |\nabla|^{1/2}f_{S}\|_{L^2} + \||\nabla|^{1/2}f_P\|_{L^2} \lesssim \|f\|_{\dot H^{1/2}}.
\end{align*}
In a similar manner, making use of \eqref{sine} we have that for any $t\in \mathbb R$
\begin{align*}
	&\| \sin(t\sqrt{-\Delta^\ast}) \sqrt{-\Delta^\ast}^{-1} g \|_{\dot H^{1/2}} \\
	&\le \| \sin(t\sqrt{-\mu\Delta}) \sqrt{-\mu\Delta}^{-1} g_S \|_{\dot H^{1/2}} + \| \sin(t\sqrt{-(\lambda+2\mu)\Delta})\sqrt{-(\lambda+2\mu)\Delta}^{-1} g_P \|_{\dot H^{1/2}} \\
	&\lesssim \| |\nabla|^{-1/2} g_{S}\|_{L^2} + \||\nabla|^{-1/2} g_P\|_{L^2} \lesssim \|g\|_{\dot H^{-1/2}},
\end{align*}
which gives \eqref{unisin}.

For \eqref{uni_inho}, by the identity $\sin\theta=(e^{i\theta}-e^{-i\theta})/2i$ and the norm equivalence \eqref{normequi},  it is enough to show that
\begin{equation*}
	\sup_{t\in\mathbb R} \bigg\| |\nabla|^{-1/2}\int_0^t e^{i(t-s)\sqrt{-\Delta^*}} \delta V(\,\cdot\,)u(\cdot,s)ds \bigg\|_{L^2} <\infty.
\end{equation*}
Since $e^{it\sqrt{-\Delta^*}}$ is an isometry in $[L^2(\mathbb R^n)]^n$, this follows from applying the following estimate
with $F=\chi_{[0,t]}\delta Vu$:
\begin{equation}\label{propa1adj}
	\bigg\| \int_{-\infty}^\infty e^{-is\sqrt{-\Delta^\ast}} F(s) ds \bigg\|_{\dot H^{-1/2}} \lesssim \|V\|^{1/2}_{\mathcal F^p}\|F\|_{L^2_{x,t}(|V|^{-1})}.
\end{equation}
Since $(e^{it\sqrt L})^\ast= e^{-it\sqrt L}$, by duality it is easy to see that
\eqref{propa1adj} is equivalent to \eqref{weihomo}.

The other assertion $\partial_tu\in C([0,\infty);\dot H^{-1/2}(\mathbb R^n))$ can be proved similarly.  Since
\begin{align} \label{u_t}
\nonumber
	\partial_tu(t)
	= \cos(t\sqrt{-\Delta^\ast})g &-\sin(t\sqrt{-\Delta^\ast})\sqrt{-\Delta^\ast} f \\
&+\int_0^t \cos\big((t-s)\sqrt{-\Delta^\ast}\big)  \delta V(\,\cdot\,)u(\,\cdot\,,s) ds
\end{align}
we proceed with the previous argument to see that
$$	\|\cos(t\sqrt{-\Delta^\ast})g\|_{\dot H^{-1/2}} \lesssim \|g\|_{\dot H^{-1/2}}, \quad \|\sin(t\sqrt{-\Delta^\ast})\sqrt{-\Delta^\ast} f\|_{\dot H^{-1/2}} \lesssim \|f\|_{\dot H^{1/2}}	$$
and
$$	\bigg\| \int_0^t \cos\big((t-s)\sqrt{-\Delta^\ast}\big)  \delta V(\cdot)u(\cdot,s) ds \bigg\|_{\dot H^{-1/2}} \lesssim |\delta|\|V\|_{\mathcal F^p}^{1/2}\|u\|_{L^2_{x,t}(|V|)}.$$

\subsubsection*{Proofs of \eqref{wei} and \eqref{smoo}}
From \eqref{sol}, Proposition \ref{prop1} and the estimate \eqref{exiinho}, it follows that
\[	\|u\|_{L^2_{x,t}(|V|)} \lesssim \|V\|_{\mathcal F^p}^{1/2}\big(\|f\|_{\dot H^{1/2}}+\|g\|_{\dot H^{-1/2}}\big) +|\delta|\|V\|_{\mathcal F^p}\|u\|_{L^2_{x,t}(|V|)}.	 \]
Thus,  for $|\delta|$ small enough, the estimate \eqref{wei} follows.

Now we prove \eqref{smoo}.
We first reduce the matter to obtaining the bounds for $(1+|\nabla|^{1/2})$ and $|\nabla|^{-1/2}$ in place of $(1-\Delta)^{1/4}$ and $(1-\Delta)^{-1/4}$ in \eqref{smoo}, respectively.
This is because $(1-\Delta)^{1/4}(1+|\nabla|^{1/2})^{-1}$ and $(1-\Delta)^{-1/4}|\nabla|^{1/2}$ are Mikhlin multipliers.
Then the reduction follows from the following boundedness of Mikhlin multipliers in Morrey Space (see, for example, Theorem 3 in \cite{MWB}):

\begin{lem}\label{Mik}
	Let $n\geq1$ and $M$ be a Mikhlin multiplier operator.
	Then for $1<p<\infty$ and $0 \leq \lambda <n$
	\begin{equation*}
	M: \mathcal{L}^{p,\lambda}(\mathbb{R}^n) \rightarrow \mathcal{L}^{p,\lambda}(\mathbb{R}^n),
	\end{equation*}	
	where $\mathcal{L}^{p,\lambda}$ denotes the Morrey space defined by
	\begin{equation*}
	\|f\|_{\mathcal{L}^{p,\lambda}} = \sup_{x \in \mathbb{R}^n,r>0} r^{-\frac{\lambda}{p}}\bigg(\int_{B(x,r)} |f(y)|^p dy \bigg)^{1/p} < \infty.
	\end{equation*}
(Note here that $\mathcal{L}^{p,\lambda}=\mathcal{F}^p$ particularly when $\lambda = n-2p$.)
\end{lem}
Indeed, applying this lemma to each component of $(1-\Delta)^{1/4}u$ in \eqref{smoo}, we have
\begin{align} \label{afM}
\nonumber
\sup_{x_0,R} \frac{1}{R} \int_{B(x_0,R)} |(1-\Delta)^{1/4}u|^2 dx
&=\sum_{j=1}^n\big\|(1-\Delta)^{1/4}(1+|\nabla|^{1/2})^{-1}(1+|\nabla|^{1/2})u_j\big\|^2_{\mathcal{L}^{2,1}}\\
\nonumber&\lesssim\sum_{j=1}^n\big\|(1+|\nabla|^{1/2})u_j\big\|^2_{\mathcal{L}^{2,1}}\\
&=\sup_{x_0,R} \frac{1}{R} \int_{B(x_0,R)}|(1+|\nabla|^{1/2})u|^2 dx.
\end{align}
Similarly,
\begin{equation}\label{afM2}
\sup_{x_0,R}\frac{1}{R}\int_{B(x_0,R)} |(1-\Delta)^{-1/4} \partial_t u|^2 dx
\lesssim\sup_{x_0,R} \frac{1}{R} \int_{B(x_0,R)}\big||\nabla|^{-1/2}\partial_t u\big|^2 dx.
\end{equation}

To bound the right side of \eqref{afM},
using H\"older's inequality, we first note that
\begin{equation*}
\sup_{x_0, R}\frac{1}{R} \int_{B(x_0,R)}|u_j|^2 dx \leq \sup_{x_0,R} \frac{1}{R} \|\chi_B\|^2_{L^{r}}\|u_j\|_{L^q}^2\lesssim \|u_j\|_{L^q}^2
\end{equation*}
provided $1/q+1/{r}=1/2$ and $2n/r=1$.
This condition $n(1/2-1/q)=1/2$ is the exact one for the Sobolev embedding $\dot H^{1/2} \hookrightarrow L^q$.
Therefore, we get
\begin{equation*}
\sup_{x_0,R} \frac{1}{R} \int_{B(x_0,R)} \int_{-T}^T |u|^2 dt dx
\lesssim T\sum_{j=1}^n\|u_j\|^2_{\dot H^{1/2}}\lesssim T\|u\|^2_{\dot H^{1/2}}.
\end{equation*}
But, $\|u\|^2_{\dot H^{1/2}}\lesssim \|f\|^2_{\dot H^{1/2}} + \|g\|^2_{\dot H^{-1/2}}$
which follows from the proof of \eqref{exi} together with \eqref{wei}.
On the other hand, by combining \eqref{sol} with the estimates \eqref{cosmoo}, \eqref{sismoo} and \eqref{smooinho} with $F=\delta Vu$,
we see that
\begin{align*}
\sup_{x_0,R}\frac{1}{R} \int_{B(x_0,R)}\int_{-T}^{T} \big||\nabla|^{1/2}u\big|^2 dtdx
&\lesssim \|f\|^2_{\dot H^{1/2}} + \|g\|^2_{\dot H^{-1/2}} + |\delta|^2 \|V\|_{\mathcal F^p} \|u\|^2_{L^2_{x,t}(|V|)} \\
&\lesssim \|f\|^2_{\dot H^{1/2}} + \|g\|^2_{\dot H^{-1/2}},
\end{align*}
where we used \eqref{wei} for the last inequality.

It remains to bound \eqref{afM2}.
By \eqref{u_t}, it suffices to show
\begin{equation} \label{u_t1}
\sup_{x_0,R} \frac{1}{R} \int_{B(x_0,R)} \int_{-T}^{T} \big| |\nabla|^{-1/2}\sqrt{-\Delta^*} \sin (t\sqrt{-\Delta^*})f \big| ^2 dx dt \lesssim \|f\|^2_{\dot H^{1/2}},
\end{equation}
\begin{equation} \label{u_t2}
\sup_{x_0,R} \frac{1}{R} \int_{B(x_0,R)} \int_{-T}^{T} \big| |\nabla|^{-1/2}\cos (t\sqrt{-\Delta^*})g \big| ^2 dx dt \lesssim \|g\|^2_{\dot H^{-1/2}}
\end{equation}
and
\begin{align}\label{u_t3}
\nonumber
\sup_{x_0,R} \frac{1}{R} \int_{B(x_0,R)} \int_{-T}^{T} \Big| |\nabla|^{-1/2}\int_0^t  \cos((t-s)\sqrt{-\Delta^*})&\delta V(\cdot)u(\cdot,s) ds\Big| ^2 dx dt \\
&\lesssim \|f\|^2_{\dot H^{1/2}}+\|g\|^2_{\dot H^{-1/2}}.
\end{align}
The left side of \eqref{u_t1} is bounded by
$$
\||\nabla|^{-1}\sqrt{-\Delta^*}f\|^2_{\dot H^{1/2}} \sim \|f\|^2_{\dot H^{1/2}}
$$
using \eqref{smoohom} with $f$ replaced by $|\nabla|^{-1}\sqrt{-\Delta^*}f$ and then the norm equivalence \eqref{normequi},
while \eqref{u_t2} follows from using \eqref{smoohom} with $f=|\nabla|^{-1}g$.
Finally, making use of \eqref{smoohom}, \eqref{propa1adj} and \eqref{wei}, the left side of \eqref{u_t3} is bounded by
\begin{align*}
\nonumber
\bigg\|\int_0^t e^{-is\sqrt{-\Delta^*}} \delta V(\cdot)u(\cdot,s) ds \bigg\|^2_{\dot H^{-1/2}}
\lesssim |\delta|^2\|V\|_{\mathcal{F}^p}\|u\|^2_{L_{x,t}^2(|V|)} \lesssim \|f\|^2_{\dot H^{1/2}} +\|g\|^2_{\dot H^{-1/2}}.
\end{align*}

\subsubsection*{Note added in the proof}
Some estimates in Sections \ref{sec3} and \ref{sec4} used in the proof were also obtained in \cite{BFPRV}
for the forward initial value problem using the spectral operational calculus.
But our approach is based on the representation of the solution in terms of Fourier multipliers,
which is simpler than the abstract one.

\section{Proof of Theorem \ref{thm1}} \label{sec6}
Finally we prove the Strichartz estimates in Theorem \ref{thm1}.
Let us recall \eqref{sol}. For the homogeneous terms we use the Helmholtz decomposition (\eqref{cosine} and \eqref{sine}),
the classical Strichartz estimate \eqref{Strr} for the free wave equation \eqref{wave_eqn} with $\delta=0$,
and the orthogonality \eqref{Pythagorean} to obtain
\begin{align}\label{Strcos}
\nonumber\|\cos(t\sqrt{-\Delta^\ast}) f\|_{L_t^q \dot H^\sigma_r} & \le \|\cos(t\sqrt{-\mu\Delta}) f_S\|_{L_t^q \dot H^\sigma_r} +\|\cos(t\sqrt{-(\lambda+2\mu)\Delta}) f_P\|_{L_t^q \dot H^\sigma_r} \\
	&\lesssim \| f_S\|_{\dot H^{1/2}} + \| f_P\|_{\dot H^{1/2}} \lesssim \|f\|_{\dot H^{1/2}}
\end{align}
and
\begin{align}\label{Strsin}
\nonumber
	\|\sin(t\sqrt{-\Delta^\ast}) \sqrt{- \Delta^\ast}^{-1} g\|_{L_t^q \dot H^\sigma_r}
\nonumber
	&\le \|\sin(t\sqrt{-\mu\Delta}) \sqrt{-\mu\Delta}^{-1} g_S\|_{L_t^q \dot H^\sigma_r} \\
\nonumber
	&\quad +\|\sin(t\sqrt{-(\lambda+2\mu)\Delta})\sqrt{-(\lambda+2\mu)\Delta}^{-1} g_P\|_{L_t^q \dot H^\sigma_r} \\
	&\lesssim \| g_S\|_{\dot H^{-1/2}} + \| g_P\|_{\dot H^{-1/2}} \lesssim \|g\|_{\dot H^{-1/2}}
\end{align}
for any wave-admissible $(q,r)$ (see \eqref{admi}).

For the inhomogeneous term, by \eqref{wei} it is enough to show
\begin{equation}\label{Str_inho}
	\|\mathcal Tu\|_{L^q_t\dot H_r^\sigma} \lesssim |\delta| \|V\|_{\mathcal F^p}^{1/2} \|u\|_{L^2_{x,t}(|V|)}.
\end{equation}
To show this, we use the identity $\sin \theta=(e^{i\theta}-e^{-i\theta})/2i$ and the Christ-Kiselev lemma (see \cite{CK}) by which it suffices to show
\begin{equation}\label{Str_goal}
	\bigg\| \int_{-\infty}^\infty e^{i(t-s)\sqrt{-\Delta^\ast}}\sqrt{-\Delta^\ast}^{-1} \delta V(\cdot) u(\cdot, s) ds \bigg\|_{L^q_t\dot H_r^\sigma} \lesssim 	|\delta| \|V\|_{\mathcal F^p}^{1/2} \|u\|_{L^2_{x,t}(|V|)}.
\end{equation}
since we are assuming $q>2$.
By the norm equivalence \eqref{normequi}, the estimate \eqref{Strsin} is equivalently written as
$$\| \sin(t\sqrt{-\Delta^\ast})f\|_{L^q_t\dot H_r^\sigma}\lesssim \|f\|_{\dot H^{1/2}}.$$
Combining this with \eqref{Strcos} gives
$$\|e^{\pm i t\sqrt{-\Delta^\ast} }f\|_{L^q_t\dot H_r^\sigma} \lesssim \|f\|_{\dot H^{1/2}}$$
from which we see that the left side of \eqref{Str_goal} is bounded by
$$	\bigg\| \int_{-\infty}^\infty e^{-is\sqrt{-\Delta^\ast}}\sqrt{-\Delta^\ast}^{-1} \delta V(\cdot) u(\cdot, s) ds \bigg\|_{\dot H^{1/2}}
	\sim \bigg\| \int_{-\infty}^\infty e^{-is\sqrt{-\Delta^\ast}}\delta V(\cdot) u(\cdot, s) ds \bigg\|_{\dot H^{-1/2}} .
$$
Now we apply the estimate \eqref{propa1adj} to conclude that the above is bounded by
$$|\delta| \|V\|_{\mathcal F^p}^{1/2} \|u\|_{L^2_{x,t}(|V|)}$$
which gives \eqref{Str_goal} as desired.
Combining \eqref{Strcos}, \eqref{Strsin} and \eqref{Str_inho} yields the Strichartz estimates \eqref{Str}, and completes the proof of Theorem \ref{thm1}.



\begin{thebibliography}{99}

\bibitem{BFRVV}
J. A. Barcel\'o, L. Fanelli, A. Ruiz, M. C. Vilela and N. Visciglia, \textit{Resolvent and Strichartz estimates for elastic wave equations},
Appl. Math. Lett. 49 (2015), 33--41.

\bibitem{BFPRV}
J. A. Barcel\'o, M. Folch-Gabayet, S. P\'erez-Esteva, A. Ruiz and M. C. Vilela, \textit{Limiting absorption principles for the Navier equation in elasticity}, Ann. Sc. Norm. Super. Pisa Cl. Sci. (5) 11 (2012), 817--842.

\bibitem{BS}
M. Beals and W. Strauss, \textit{$L^p$ estimates for the wave equation with a potential}, Comm. Partial Differential Equations 18 (1993), 1365--1397.

\bibitem{BPST-Z}
N. Burq, F. Planchon, J. G. Stalker and A. S. Tahvildar-Zadeh, \textit{Strichartz estimates for the wave and Schr\"odinger equations with the inverse-square potential}, J. Funct. Anal. 203 (2003), 519--549.

\bibitem{BPST-Z2}
N. Burq, F. Planchon, J. G. Stalker and A. S. Tahvildar-Zadeh, \textit{Strichartz estimates for the wave and Schr\"odinger equations with potentials of critical decay}, Indiana Univ. Math. J. 53 (2004), 1665--1680.

\bibitem{CF}
F. Chiarenza and M. Frasca, \textit{A remark on a paper by C. Fefferman}, Proc. Amer. Math. Soc. 108 (1990), 407--409.

\bibitem{CK}
M. Christ and A. Kiselev, \textit{Maximal functions associated to filtrations}, J. Funct. Anal. 179 (2001), 409--425.

\bibitem{CR}
R. Coifman and R. Rochberg, \textit{Another characterization of BMO}, Proc. Amer. Math. Soc. 79 (1980), 249--254.

\bibitem{Co}
L. Cossetti, \textit{Bounds on eigenvalues of perturbed Lam\'e operators with complex potentials}, Preprint, arXiv:1904.08445.

\bibitem{C}
S. Cuccagna, \textit{On the wave equation with a potential}, Comm. Partial Differential Equations 25 (2000), 1549--1565.

\bibitem{D}
P. D'Ancona, \textit{On large potential perturbations of the Schr\"odinger, wave and Klein-Gordon equations}, Commun. Pure Appl. Anal. 19 (2020), 609--640.

\bibitem{GV}
V. Georgiev and N. Visciglia, \textit{Decay estimates for the wave equation with potential}, Comm. Partial Differential Equations 28 (2003), 1325--1369.

\bibitem{GVV}
M. Goldberg, L. Vega and N. Visciglia, \textit{Counterexamples of Strichartz inequalities for Schr\"odinger equations with repulsive potentials}, Int. Math. Res. Not. 2006, Art. ID 13927, 16pp.

\bibitem{K}
K. Yajima, \textit{The $W^{k,p}$-continuity of wave operators for Schr\"odinger operators}, J. Math. Soc. Japan 47 (1995), 551--581.

\bibitem{KT}
M. Keel and T. Tao, \textit{Endpoint Strichartz estimates}, Amer. J. Math. 120 (1998), 955--980.

\bibitem{KSS}
S. Kim, I. Seo and J. Seok, \textit{Note on Strichartz inequalities for the wave equation with potential}, Math. Inequal. Appl. 23 (2020), 377--382.

\bibitem {KM}
S. Klainerman and M. Machedon, \textit{Space-time estimates for null forms and the local existence theorem}, Comm. Pure Appl. Math. 46 (1993), 1221--1268.

\bibitem{LL}
L. D. Landau and E. M. Lifshitz, \textit{Theory of Elasticity}, Pergamon, 1970.

\bibitem{LS}
H. Lindblad and C. D. Sogge, \textit{On existence and scattering with minimal regularity for semilinear wave equations}, J. Funct. Anal. 130 (1995), 357--426.

\bibitem{MWB}
D. Maharani, J. Widjaja and M. Wono Setya Budhi, \textit{Boundedness of Mikhlin Operator in Morrey Space}, J. Phys.: Conf. Ser. 1180 (2019), 012002.

\bibitem{MH}
J. E. Marsden and T. J. R. Hughes, \textit{Mathematical foundations of elasticity}, Prentice Hall, 1983, reprinted by Dover Publications, N.Y., 1994.

\bibitem{P}
S. Petermichl, \textit{The sharp weighted bound for the Riesz transforms}, Proc. Amer. Math. Soc. 136 (2008), 1237--1249.

\bibitem{PST-Z}
F. Planchon, J. G. Stalker and A. S. Tahvildar-Zadeh, \textit{$L^p$ estimates for the wave equation with the inverse-square potential}, Discrete Contin. Dyn. Syst. 9 (2003), 427--442.

\bibitem{RV}
A. Ruiz and L. Vega, \textit{Local regularity of solutions to wave equations with time-dependent potentials}, Duke Math. J. 76 (1994), 913--940.

\bibitem{Sohr}
H. Sohr, \textit{The Navier-Stokes equations. An elementary functional analytic approach}, [2013 reprint of the 2001 original] Modern Birkh\"auser Classics. Birkh{\"a}user/Springer Basel AG, Basel, 2001.

\bibitem{S}
E. Stein, \textit{Harmonic analysis: real-variable methods, orthogonality, and oscillatory integrals}, Princeton University Press, Princeton, NJ, 1993.

\bibitem{Str}
R. S. Strichartz, \textit{Restrictions of Fourier transforms to quadratic surfaces and decay of solutions of wave equations}, Duke Math. J. 44 (1977), 705--714.


\end{thebibliography}
\end{document}